\documentclass[reqno]{amsart}

\usepackage[latin1]{inputenc}
\usepackage{amssymb}
\usepackage{graphicx}
\usepackage{amscd}
\usepackage[hidelinks]{hyperref}
\usepackage{color}
\usepackage{float}
\usepackage{graphics,amsmath,amssymb}
\usepackage{amsthm}
\usepackage{amsfonts}
\usepackage{latexsym}
\usepackage{epsf}
\usepackage{enumerate}
\usepackage{xifthen}
\usepackage{mathrsfs}
\usepackage{dsfont}
\usepackage{makecell}
\usepackage[FIGTOPCAP]{subfigure}
\usepackage{amsmath}
\allowdisplaybreaks[4]
\usepackage{listings}
\usepackage{etoolbox}
\usepackage{fancyhdr}

 \newtheoremstyle{mytheorem}
 {3pt}
 {3pt}
 {\slshape}
 {}
 {\bfseries}
 {.}
 { }
 {}

\numberwithin{equation}{section}

\theoremstyle{theorem}
\newtheorem{theorem}{Theorem}[section]

\newtheorem{lemma}[theorem]{Lemma}

\theoremstyle{definition}

\renewcommand{\MR}[1]{\href{http://www.ams.org/mathscinet-getitem?mr=#1}{MR#1}}

\newcommand{\Keywords}[1]{\ifthenelse{\isempty{#1}}{}{\smallskip \smallskip \noindent \textbf{Keywords}. #1}}
\newcommand{\MSC}[2][2010]{\ifthenelse{\isempty{#2}}{}{\smallskip \smallskip \noindent \textbf{#1MSC}. #2}}
\newcommand{\abstractnote}[1]{\ifthenelse{\isempty{#1}}{}{\smallskip \smallskip \noindent \textsuperscript{\dag}#1}}

\makeatletter
\def\specialsection{\@startsection{section}{1}%
  \z@{\linespacing\@plus\linespacing}{.5\linespacing}%
  {\normalfont}}
\def\section{\@startsection{section}{1}%
  \z@{.7\linespacing\@plus\linespacing}{.5\linespacing}%
  {\normalfont\scshape}}
\patchcmd{\@settitle}{\uppercasenonmath\@title}{\Large\boldmath}{}{}
\patchcmd{\@settitle}{\begin{center}}{\begin{flushleft}}{}{}
\patchcmd{\@settitle}{\end{center}}{\end{flushleft}}{}{}
\patchcmd{\@setauthors}{\MakeUppercase}{\normalsize}{}{}
\patchcmd{\@setauthors}{\centering}{\raggedright}{}{}
\patchcmd{\section}{\scshape}{\large\bfseries\boldmath}{}{}
\patchcmd{\subsection}{\bfseries}{\bfseries\boldmath}{}{}
\renewcommand{\@secnumfont}{\bfseries}
\patchcmd{\@startsection}{\@afterindenttrue}{\@afterindentfalse}{}{}
\patchcmd{\abstract}{\leftmargin3pc}{\leftmargin1pc}{}{}

\def\maketitle{\par
  \@topnum\z@ 
  \@setcopyright
  \thispagestyle{empty}
  \ifx\@empty\shortauthors \let\shortauthors\shorttitle
  \else \andify\shortauthors
  \fi
  \@maketitle@hook
  \begingroup
  \@maketitle
  \toks@\@xp{\shortauthors}\@temptokena\@xp{\shorttitle}%
  \toks4{\def\\{ \ignorespaces}}
  \edef\@tempa{%
    \@nx\markboth{\the\toks4
      \@nx\MakeUppercase{\the\toks@}}{\the\@temptokena}}%
  \@tempa
  \endgroup
  \c@footnote\z@
  \@cleartopmattertags
}
\makeatother

\title[A curious identity]{A curious identity and its applications to partitions with bounded part differences}

\author[S. Chern]{Shane Chern}
\address{Department of Mathematics, The Pennsylvania State University, University Park, PA 16802, USA}
\email{shanechern@psu.edu}

\date{}

\begin{document}

{\footnotesize\noindent \textit{New Zealand J. Math.} \textbf{47} (2017), 23--26. \MR{3691619}.}

\bigskip \bigskip

\maketitle

\begin{abstract}

In this note, we present a curious $q$-series identity with applications to certain partitions with bounded part differences.


\MSC{Primary 05A17; Secondary 11P84.}
\end{abstract}

\section{Introduction}

A \textit{partition} of a positive integer $n$ is a non-increasing sequence of positive integers whose sum is $n$. Recently, motivated by the work of Andrews, Beck and Robbins \cite{ABR2015}, Breuer and Kronholm \cite{BK2016} obtained the generating function of partitions where the difference between largest and smallest parts is at most a fixed positive integer $t$,
\begin{equation}\label{eq:BK}
\sum_{n\ge 1}p_t(n) q^n=\frac{1}{1-q^t}\left(\frac{1}{(q;q)_t}-1\right),
\end{equation}
where $p_t(n)$ denotes the number of such partitions of $n$. Here and in what follows, we use the standard $q$-series notation
$$(a;q)_n:=\prod_{k=0}^{n-1} (1-a q^k),\quad \text{for $|q|<1$}.$$

Subsequently, the author and Yee \cite{Che2017,CY2017} considered an overpartition analogue of Breuer and Kronholm's result. Here an \textit{overpartition} of $n$ is a partition of $n$ where the first occurrence of each distinct part may be overlined. Let $g_t(m,n)$ count the number of overpartitions of $n$ in which there are exactly $m$ overlined parts,  the difference between largest and smallest parts  is at most $t$, and if the difference between largest and smallest parts is exactly $t$, then the largest parts cannot be overlined. The author and Yee proved
\begin{equation}\label{eq:CY}
\sum_{n\ge 1}\sum_{m\ge 0} g_t(m,n)z^mq^n=\frac{1}{1-q^t}\left(\frac{(-zq;q)_t}{(q;q)_t}-1\right).
\end{equation}

Suggested by George E. Andrews, it is also natural to study other types of partitions with bounded part differences. Let $pd_t(n)$ (resp. $po_t(n)$) count the number of partitions of $n$ in which all parts are distinct (resp. odd) and the difference between largest and smallest parts is at most $t$.
\begin{theorem}
We have
\begin{equation}\label{eq:pdt}
\sum_{n\ge 1}pd_t(n) q^n=\frac{1}{1-q^{t+1}}\left((-q;q)_{t+1}-1\right),
\end{equation}
and
\begin{equation}\label{eq:pot}
\sum_{n\ge 1}po_{2t}(n) q^n=\frac{1}{1-q^{2t}}\left(\frac{1}{(q;q^2)_t}-1\right).
\end{equation}
\end{theorem}

Noting that \eqref{eq:BK}--\eqref{eq:pot} have the same flavor, we therefore want to seek for a unified proof of these generating function identities. 

Let $t$ be a fixed positive integer. Assume that $\alpha$, $\beta$, $q$ are complex variables with $|q|<1$, $q\ne 0$, $\alpha\ne \beta q$ and $(\beta q;q)_t\ne 0$. We define the following sum
\begin{equation}
S(\alpha,\beta;q;t):=\sum_{r\ge 1}\frac{(1-\alpha q^{r})(1-\alpha q^{r+1})\cdots(1-\alpha q^{r+t-2})}{(1-\beta q^{r})(1-\beta q^{r+1})\cdots(1-\beta q^{r+t})}q^r.
\end{equation}
The following curious identity provides such a unified approach.

\begin{theorem}\label{th:main}
We have
\begin{equation}\label{eq:main}
S(\alpha,\beta;q;t)=\frac{q}{(\beta q-\alpha)(1-q^t)}\left(\frac{(\alpha;q)_t}{(\beta q;q)_t}-1\right).
\end{equation}
\end{theorem}

\section{Proof of Theorem \ref{th:main}}
Let
$${}_{r+1}\phi_r\left(\begin{matrix} a_0,a_1,a_2\ldots,a_r\\ b_1,b_2,\ldots,b_r \end{matrix}; q, z\right):=\sum_{n\ge 0}\frac{(a_0;q)_n(a_1;q)_n\cdots(a_r;q)_n}{(q;q)_n(b_1;q)_n\cdots (b_r;q)_n} z^n.$$
The following two lemmas are useful in our proof.

\begin{lemma}[First $q$-Chu--Vandermonde Sum {\cite[Eq. (17.6.2)]{And2010}}]\label{le:chu}
We have
\begin{equation}\label{eq:chu}
{}_{2}\phi_{1}\left(\begin{matrix} a,q^{-n}\\ c \end{matrix}; q, \frac{cq^{n}}{a}\right)=\frac{(c/a;q)_n}{(c;q)_n}.
\end{equation}
\end{lemma}

\begin{lemma}[$q$-Analogue of the Kummer--Thomae--Whipple Transformation {\cite[p. 72, Eq. (3.2.7)]{GR2004}}]\label{le:32}
We have
\begin{equation}\label{eq:32}
{}_{3}\phi_{2}\left(\begin{matrix} a,b,c\\ d,e \end{matrix}; q, \frac{de}{abc}\right)=\frac{(e/a;q)_\infty(de/bc;q)_\infty}{(e;q)_\infty(de/abc;q)_\infty}{}_{3}\phi_{2}\left(\begin{matrix} a, d/b, d/c\\ d, de/bc \end{matrix}; q, \frac{e}{a}\right).
\end{equation}
\end{lemma}

\begin{proof}[Proof of Theorem \ref{th:main}]
We have
\begin{align}
&S(\alpha,\beta;q;t)\\
&\quad=\sum_{r\ge 1}\frac{(1-\alpha q^{r})(1-\alpha q^{r+1})\cdots(1-\alpha q^{r+t-2})}{(1-\beta q^{r})(1-\beta q^{r+1})\cdots(1-\beta q^{r+t})}q^r \notag\\
&\quad=\sum_{r\ge 1}\frac{(\alpha;q)_{r+t-1}(\beta;q)_{r}}{(\alpha;q)_{r}(\beta;q)_{r+t+1}}q^r \notag\\
&\quad =\sum_{r\ge 0}\frac{(\alpha;q)_{r+t}(\beta;q)_{r+1}}{(\alpha;q)_{r+1}(\beta;q)_{r+t+2}}q^{r+1} \notag\\
&\quad=\frac{q(1-\beta)(\alpha;q)_t}{(1-\alpha)(\beta;q)_{t+2}}\sum_{r\ge 0}\frac{(q;q)_r (\beta q;q)_{r} (\alpha q^t;q)_{r}}{(q;q)_r(\alpha q;q)_{r}(\beta q^{t+2};q)_{r}}q^r \notag \\
&\quad=\frac{q(\alpha q;q)_{t-1}}{(\beta q;q)_{t+1}}\ {}_{3}\phi_{2}\left(\begin{matrix} q,\beta q,\alpha q^{t}\\ \alpha q,\beta q^{t+2} \end{matrix}; q, q\right) \notag \\
&\quad = \frac{q(\alpha q;q)_{t-1}}{(\beta q;q)_{t+1}}\frac{(\beta q^{t+1};q)_\infty(q^2;q)_\infty}{(\beta q^{t+2};q)_\infty(q;q)_\infty}\ {}_{3}\phi_{2}\left(\begin{matrix} q,\alpha/\beta,q^{1-t}\\ \alpha q,q^{2} \end{matrix}; q, \beta q^{t+1}\right) \tag{by Eq. \eqref{eq:32}}\\
&\quad=  \frac{q(\alpha q;q)_{t-1}}{(1-q)(\beta q;q)_{t}} \sum_{r\ge 0}\frac{(\alpha/\beta;q)_{r}(q^{1-t};q)_{r}}{(\alpha q;q)_{r}(q^2;q)_r}\left(\beta q^{t+1}\right)^r \notag\\
&\quad=   \frac{q(\alpha q;q)_{t-1}}{(1-q)(\beta q;q)_{t}} \frac{(1-\alpha)(1-q)}{\beta q^{t+1}\left(1-\frac{\alpha}{\beta q}\right)(1-q^{-t})}\sum_{r\ge 0}\frac{\left(\frac{\alpha}{\beta q};q\right)_{r+1}(q^{-t};q)_{r+1}}{(\alpha;q)_{r+1}(q;q)_{r+1}}\left(\beta q^{t+1}\right)^{r+1} \notag\\
&\quad= \frac{q}{(\beta q-\alpha)(q^t-1)}\frac{(\alpha;q)_t}{(\beta q;q)_t} \left({}_{2}\phi_{1}\left(\begin{matrix} \frac{\alpha}{\beta q},q^{-t}\\ \alpha \end{matrix}; q, \beta q^{t+1}\right)-1\right) \notag \\
&\quad= \frac{q}{(\beta q-\alpha)(q^t-1)}\frac{(\alpha;q)_t}{(\beta q;q)_t} \left(\frac{(\beta q;q)_t}{(\alpha;q)_t}-1\right) \tag{by Eq. \eqref{eq:chu}}\\
&\quad= \frac{q}{(\beta q-\alpha)(1-q^t)}\left(\frac{(\alpha;q)_t}{(\beta q;q)_t}-1\right). \notag
\end{align}
\end{proof}

\section{Applications}

We now show how Theorem \ref{th:main} may prove \eqref{eq:BK}--\eqref{eq:pot}.

At first, we prove the two new identities \eqref{eq:pdt} and \eqref{eq:pot}. Note that the generating function for partitions counted by $pd_t(n)$ with smallest part equal to $r$ is
$$q^r(1+q^{r+1})(1+q^{r+2})\cdots(1+q^{r+t}).$$
Hence
\begin{align*}
\sum_{n\ge 1}pd_t(n) q^n=\sum_{r\ge 1}(1+q^{r+1})(1+q^{r+2})\cdots(1+q^{r+t})q^r=S(-q,0;q;t+1).
\end{align*}
It follows by Theorem \ref{th:main} that
$$\sum_{n\ge 1}pd_t(n) q^n=S(-q,0;q;t+1)=\frac{1}{1-q^{t+1}}\left((-q;q)_{t+1}-1\right).$$

To see \eqref{eq:pot}, one readily verifies that the generating function for partitions counted by $po_{2t}(n)$ with smallest part equal to $2r-1$ is
$$\frac{q^{2r-1}}{(1-q^{2r-1})(1-q^{2r+1})\cdots(1-q^{2r+2t-1})}.$$
Hence
\begin{align*}
\sum_{n\ge 1}po_{2t}(n) q^n&=\sum_{r\ge 1}\frac{1}{(1-q^{2r-1})(1-q^{2r+1})\cdots(1-q^{2r+2t-1})}q^{2r-1}\\
&=q^{-1}S(0,q^{-1};q^2;t)=\frac{1}{1-q^{2t}}\left(\frac{1}{(q;q^2)_t}-1\right).
\end{align*}
Here the last equality follows again from Theorem \ref{th:main}. We remark that for any positive integer $t$, $po_{2t}(n)=po_{2t+1}(n)$ since only odd parts are allowed in this case. Hence it suffices to consider merely the generating function of $po_{2t}(n)$.

The proofs of \eqref{eq:BK} and \eqref{eq:CY} are similar. We omit the details here.

\subsection*{Acknowledgements}

I would like to thank the referee for careful reading and useful comments.

\bibliographystyle{amsplain}

\end{document}